\documentclass[12pt]{amsart}
\makeatletter
\def\@citecolor{blue}
\def\@urlcolor{blue}
\def\@linkcolor{blue}
\makeatother

\usepackage[headings]{fullpage}
\usepackage{lmodern}
\usepackage[T1]{fontenc}
\usepackage{amsmath,amsthm,amssymb,amsfonts,latexsym,mathrsfs}
\usepackage[all]{xy}
\usepackage{paralist}
\usepackage{color}
\usepackage{ifthen} 
\usepackage[normalem]{ulem}

\makeatletter

\@addtoreset{equation}{section}
\def\theequation{\thesection.\@arabic \c@equation}
\def\@citecolor{blue}
\def\@urlcolor{blue}
\def\@linkcolor{blue}
\def\theenumi{\@roman\c@enumi}

\makeatother

\theoremstyle{plain}
\newtheorem{theorem}[equation]{Theorem}
\newtheorem{lemma}[equation]{Lemma}
\newtheorem{corollary}[equation]{Corollary}
\newtheorem{proposition}[equation]{Proposition}

\newtheorem{notation}[equation]{Notation}
\theoremstyle{definition}
\newtheorem{remark}[equation]{Remark}

\newtheorem{remarks}[equation]{Remarks}

\newtheorem{examples}[equation]{Examples}

\newtheorem{definition}[equation]{Definition}

\def\NZQ{\mathbb}               
\def\NN{{\NZQ N}}
\def\ZZ{{\NZQ Z}}
\def\RR{{\NZQ R}}
\def\opn#1#2{\def#1{\operatorname{#2}}} 
\opn\supp{supp}
\opn\chara{char}
\opn\length{\ell}
\opn\projdim{proj\,dim}
\opn\depth{depth}
\opn\reg{reg}
\opn\lreg{lreg}
\opn\sat{^{sat}}
\opn\lex{^{lex}}
\opn\gin{gin}
\opn\GL{GL}
\opn\Gf{Gf}
\opn\rk{rank}
\opn\cs{{\bf cs}}
\opn\gcs{{\bf gcs}}
\opn\SSK{SSk}
\opn\Ker{Ker}
\opn\Coker{Coker}
\opn\Im{Im}
\opn\Hom{Hom}
\opn\Tor{Tor}
\opn\Ext{Ext}
\opn\End{End}
\opn\Aut{Aut}
\opn\id{id}
\opn\GL{GL}
\let\:=\colon

\opn\Gin{Gin}

\opn\Hilb{Hilb}
\opn\ini{in}
\opn\End{end}

\begin{document}

\title{Generic circuits sets and general initial ideals with respect to weights}
\author{Giulio Caviglia} 
\thanks{The work of the first author was supported by a grant from the Simons Foundation (209661 to G. C.)}
\address{Giu\-lio Ca\-vi\-glia - Department of Mathematics -  Purdue University - 150 N. University Street, West Lafayette - 
  IN 47907-2067 - USA}
\email{gcavigli@math.purdue.edu}
\author{Enrico Sbarra}
\address{Enrico Sbarra - Dipartimento di Matematica - Universit\`a degli Studi di Pisa - Largo Bruno Pontecorvo 5 - 56127 Pisa - Italy}
\email{sbarra@dm.unipi.it}

\begin{abstract}
 We study the set of circuits of a homogeneous ideal and that of its truncations, and  introduce the notion of  generic circuits set. We show how this is a well-defined invariant that can be used, in the case of initial ideals with respect to weights, as a counterpart of the (usual) generic initial ideal with respect to monomial orders. As an application we recover the existence of the generic fan introduced by R\"omer and Schmitz for studying generic tropical varieties. We also consider general initial ideals with respect to weights and  show, in analogy to the fact that generic initial ideals are Borel-fixed, that these are  fixed under the action of certain Borel subgroups of the general linear group.
\end{abstract}

\date{\today}

\maketitle

\section*{Introduction}
In the study of homogeneous ideals in a polynomial ring it is a standard technique  to pass to initial ideals. Also, in order to work with a monomial ideal more closely related to a given homogeneous ideal $I$, i.e., with a monomial ideal which  shares with $I$  important numerical invariants other than  the Hilbert function, one can choose to work in generic coordinates or, in other words, to consider a generic initial ideal of $I$ with respect to some monomial order. Even though some of the ideas underlying  the notion of generic initial ideal were already present in the works of Hartshorne \cite{Ha} and Grauert \cite{Gt}, a proper definition  as well as the study of some of its main properties is to be found only later in the work of Galligo \cite{Ga}, where  characteristic zero is assumed, and the subsequent  paper of Bayer and Stillman \cite{BaSt}, where the assumption on the characteristic is dropped. It is also shown there that generic initial ideals are invariant under the action of the Borel subgroup of the general linear group of coordinates changes $\GL_n(K)$; thus, they are endowed with interesting combinatorial properties which depend on the characteristic, but are well understood (see for instance \cite{Pa}, also for the study of other group actions). More can be said if one considers generic initial ideals with respect  to some special monomial orders such as the lexicographic and the reverse-lexicographic orders. The generic initial ideal of $I$ with respect to the revlex order has the same depth as $I$; therefore, it has the same projective dimension as $I$ and its quotient ring is Cohen-Macaulay exactly when that of $I$ is Cohen-Macaulay; furthermore, it shares with $I$ the same Castelnuovo-Mumford regularity and, in general, the same positions and values of extremal Betti numbers, cf. \cite{BaSt} and \cite{BaChPo}. On the other hand, when one considers the lex order, the  generic initial ideal of $I$ captures other geometric invariants of the projective variety defined by $I$, see for instance \cite{Gr} Section 6, \cite{CoSi} and \cite{AhKwSo}.
The interested reader is referred to the standard references \cite{Gr} and \cite{Ei} and will also find the dedicated parts of Herzog and Hibi's book \cite{HeHi} useful to understand the connection with extremal Betti numbers and with shifting operations.

The main question we address in this paper is the following: {\em How can one define the generic initial ideal with respect to a weight?} The initial ideal with respect to a weight $\omega$ of $gI$ is not necessarily constant on a non-empty Zariski open subset of $\GL_n(K)$ as, for instance, if $\omega=(1,\ldots,1)$ then $\ini_\omega(gI)=gI$ for all coordinates changes $g$. We provide an answer to the above question by introducing some new invariants of $I$.\\

This paper is organized as follows. The first  section is dedicated to introducing some notation and recalling some well-known properties of monomial orders, initial ideals with respect to weights, reduced and universal Gr\"obner bases.
In the second section, Definitions \ref{cs} and \ref{gcs}, we introduce the notion of circuits set and generic circuits set, we explain their basic properties, cf. Lemma \ref{useful} and Theorem \ref{generalcs},  and we relate the circuits set and the generic circuits set of a homogeneous ideal $I$ to its reduced Gr\"obner bases and  Gr\"obner fan, cf. Lemma \ref{pretim}, Corollaries \ref{cervellobollito} and \ref{preTim}. As an application, we recover in Corollary \ref{tim} one of the main results of  \cite{R\"oSc}, Corollary 3.2, where the {\em generic Gr\"obner fan} is introduced. 
In the third and last section we explain why what would be the natural definition of generic initial ideal with respect to weights  does not return an invariant of a homogeneous ideal, and we suggest what provides, in our opinion,  valid alternatives to  generic initial ideals when working with weights: generic and general circuits set, their truncations,  their initial circuits set and general initial ideals. Finally, in analogy with  what  is known in the case of monomial orders,  we show that general initial ideals  with respect to weights are stable under the action of certain subgroups of the general linear group.\\

We would like to dedicate this paper to J\"urgen Herzog, teacher and collaborator, to his ability for sharing his passion for Commutative Algebra with so many students all over the world.

\section{Notation and preliminaries}
 Let $A=K[X_1,\ldots,X_n]$ be a polynomial ring over a field $K$.  Given a non-zero polynomial $f\in A$, we write it uniquely as sum of monomials with non-zero coefficients  and we call the set of all such monomials, denoted by $\supp(f)$, the {\em (monomial)  support} of $f$. When $S\subseteq A$ is a set, $\supp(S)$ will be the set of all monomial supports of the  polynomials in $S$. When $I$ is a homogeneous ideal of $A$ and $d$ an integer, by $I_d$ and $I_{\leq d}$ we denote the degree $d$ part of $I$ and  the {\em truncation} $\oplus_{j\leq d}I_j$ of $I$ at (and below) $d$ respectively. 

\subsection{Monomial orders} We recall that a  {\em monomial order} on $A$ is a total order $\prec$ on the monomials of $A$ which is also compatible with multiplication, i.e., for all ${\bf X^a}, {\bf X^b}, {\bf X^c}$ monomials of $A$ with  ${\bf X^c} \not = 1$ and ${\bf X^a}\prec {\bf X^b}$ one has   ${\bf X^a} \prec {\bf X^a}{\bf X^c}\prec {\bf X^b}{\bf X^c}.$ 
Given a monomial order $\prec$, we denote by $\ini_\prec(f)$  the greatest with respect to $\prec$ monomial  in $\supp(f)$; we call it the {\em initial monomial} (or {\em leading monomial}) of $f.$  Accordingly, given a homogeneous ideal $I$, we call {\em initial ideal} of $I$ with respect to $\prec$ and denote it by $\ini_\prec(I)$, the ideal generated by all the initial monomials of elements of $I$.
A finite set $G=\{f_1,\ldots,f_r\}$ of elements of $I$ such that $\{\ini_\prec(f_1),\ldots,\ini_\prec(f_r)\}$ is a set of  generators for $\ini_\prec(I)$ is called a {\em Gr\"obner basis} of $I$ (with respect to $\prec$).
 Furthermore, if $f_1,\ldots,f_r$ are monic (with respect to $\prec$) and $\ini_\prec(f_i)$ does not divide any monomial in  $\supp (f_j)$ for $i\neq j$ then we call $G$ the {\em reduced Gr\"obner basis} of $I$ (with respect to $\prec$). It is not difficult to see that such a basis always exists and it is uniquely determined by $\prec$ and $I$; moreover, if $\prec$ and $\prec'$ are two monomial orders such that  $\ini_{\prec}(I)=\ini_{\prec'}(I)$ then the reduced Gr\"obner bases of $I$ with respect to $\prec$ and to $\prec'$ are the same. 
\noindent 
It is well-known that a given homogeneous ideal has only a finite number of initial ideals; therefore, it has finitely many reduced Gr\"obner bases. A subset $G$ of $A$ is called a {\em universal Gr\"obner basis} of $I$ if $G$ is a Gr\"obner basis of $I$ with respect to all monomial orders simultaneously. Such a basis can be obtained, for instance, as the union of all the reduced Gr\"obner bases of $I$ (cf. \cite{St}, Corollary 1.3) in which case we  call it  the {\em canonical universal  Gr\"obner basis} of $I$.
 
 \subsection{Initial ideals with respect to weights.}  We now consider the general case of initial ideals defined by using weights and  summarize some of the  basic properties and constructions; our main references are, as before, the books of Eisenbud \cite{Ei} and Sturmfels \cite{St}. We shall call a vector of $\RR^n$ a {\em weight vector} or, simply, a {\em weight}. Given a polynomial $f=\sum_i\alpha_i{\bf X^{a_i}}$, one lets the {\em initial form} of $f$ with respect to $\omega$ be the sum of all terms  $\alpha_{j}{\bf X^{a_{j}}}$ of $f$ which have maximal weight, i.e., such that the scalar product $\omega\cdot{\bf {a_{j}}}$  is maximal. Accordingly, one defines the {\em initial ideal with respect to $\omega$} of a given ideal $I$ as the ideal $\ini_\omega(I)$ generated by all the initial forms of polynomials in $I.$ This ideal will not be monomial in general. Similarly, one defines  $\ini_\omega(W)$ for a $K$-vector subspace $W$ of $A$.

\noindent
Let now $\prec$ be a monomial order; it is natural to define a new monomial order $\prec_\omega$ by refining $\omega$ by means of $\prec$, so that for all $f\in A$ 
one has $\ini_{\prec_\omega}(f)=\ini_\prec(\ini_\omega(f))$ and, similarly, for all ideals $I\subseteq A$ one has $\ini_{\prec_\omega}(I)=\ini_\prec(\ini_\omega(I))$, which also yields that $I$ and $\ini_\omega(I)$ share the same Hilbert function. Furthermore, if $G=\{f_1,\ldots,f_r\}$ is a (reduced, universal) Gr\"obner basis of $I$ with respect to $\prec_\omega$, then $\{\ini_\omega(f_i) \: i=1,\ldots,r \}$ is a (reduced, universal) Gr\"obner basis of $\ini_\omega(I)$ with respect to $\prec$. The use of weights generalizes monomial orders also in  the following sense: for any monomial order $\prec$ and any homogeneous ideal $I$, there exists a non-negative integral weight $\omega$ such that $\ini_{\prec}(I)=\ini_\omega(I)$, by \cite{St}, Proposition 1.11 (see also \cite{Ro}). We also observe that, if $I\subseteq A$ is a homogeneous ideal, $\prec$ is a monomial order and   $\omega, \omega'\in\RR^n$ are such that $\ini_\omega(\cdot)$ and $\ini_{\omega'}(\cdot)$ coincide on all elements of  a reduced Gr\"obner basis of $I$ with respect to  $\prec_\omega$, then the initial ideal of $I$ with respect to $\omega$ and $\omega'$ coincide, since both have  the same Hilbert function as $I$.

\subsection{A flat family argument}\label{flatt}
We would like to conclude this section with a technical observation we shall need later  when we use a classical flat family argument, as of \cite{Ei} Theorem 15.17. Let $\omega \in\ZZ^n$, $I$ a given ideal of $A$ and $A[t]$ a polynomial ring over $A$. Let ${\bf e_i}$ denote the $i^{\rm th}$ element of the standard basis of $\ZZ^n$; for all $f=\sum_i\alpha_i{\bf X^{a_i}}\in I$ we denote the homogenization $t^{\max_i\{\omega\cdot {\bf a_i}\}}f(t^{-\omega\cdot{\bf e_1}}X_1,\ldots,t^{-\omega\cdot{\bf e_n}}X_n)$ of $f$ with respect to $\omega$ by $\tilde{f}$; also, $\tilde{I}$ will denote the ideal of $A[t]$ generated by all $\tilde{f}$ with $f\in I$. The ideal $\tilde{I}$ is the homogenization of $I$ with respect to $\omega$.
 One can thus build a  family $\{I_{a} \: a\in K\}$ of ideals of $A$, where $I_a=\tilde{I}_{t=a}$ is the ideal $\tilde{I}$ evaluated at $t=a$. It is important to notice that $I_1=I$, $I_0= \ini_{\omega}(I)$,  and that for all $a\not =0$ the ideal $I_{a}$ is the image of $I$ under the diagonal change of coordinates $D_a$ which maps $X_i$ to $a^{- \omega\cdot{\bf e_i}}X_i.$ This family is flat because the Hilbert function is constant on its elements, see also Definition 1.17 and Theorem 1.18 in \cite{Gr}.

\section{Generic circuits sets}
In this section we define the notion of circuits set and generic circuits set of homogeneous ideals and  show how to use these definitions to compute reduced and  universal Gr\"obner bases, and Gr\"obner fans.

\begin{definition}\label{cs}
Let $I$ be a subset of $A$. We define the {\em circuits set of $I$}, denoted  by $\cs(I)$,  to be the set of all minimal (with respect to inclusion) elements of $\supp(I)$. We say that a set $T$ is a {\em circuits set} if $T=\cs(I)$ for some subset $I$ of  $A$. In particular, $T$ is a collection of finite sets of monomials of $A$. 

\end{definition}

\noindent
The name we chose in the above definition  comes from Matroid Theory, see \cite{Ox} or any other standard reference: a circuit in a matroid  is a minimal dependent subset, i.e  a dependent set whose proper subsets are all independent. When $I$ is a $K$-vector subspace of $A$, as it is for instance when $I$ is a homogeneous ideal, one can define a matroid by considering the set $S$ of all monomials of $A$ and declaring a subset of $S$ independent (resp. dependent) if its image in $A/I$ consists of linearly independent (resp. dependent) elements. The support of a polynomial $f\in I$ is minimal among all the supports of elements of $I$ if and only if it is a circuit in the  above  matroid.
It is immediately seen that if $I$ is finite then so is $\cs(I)$. Moreover, if $I$ is a homogeneous ideal of $A$ then $\cs(I)= \sqcup_{d} \cs(I_d)$ and $\cs(I_{\leq d})= \sqcup_{h\leq d} \cs(I_h)$. Also, we notice that if $I$ is a monomial ideal then $\cs(I)$ is just the set of all monomials in $I.$ Clearly, when $I$ and $J$ are monomial ideals then  $\cs(I)=\cs(J)$ if and only if $I=J$, a fact which is false in general, e.g., in $A=K[X_1, X_2]$, where $\chara(K)\not =2$, the ideals $(X_1+X_2)+(X_1,X_2)^2$, $(X_1-X_2)+(X_1,X_2)^2$ are distinct and have same circuits sets.
\noindent
It is useful to point out that if $\{f_1,\ldots,f_r\}$ is a  Gr\"obner basis with respect to a monomial order $\prec$, then $\supp(f_i)$ is not necessarily an element of $\cs(I)$, take for instance $A=K[X_1,X_2]$ and $I=(X_1+X_2,X_2)$. 
In fact, if $\{f_1,\ldots,f_r\}$ is a reduced Gr\"obner basis then $\supp(f_i)\in \cs(I)$ for all $i=1,\ldots,r$: If $\supp(f_h)\not\in\cs(I)$ for some $h$, then there would exist  $g\in I$ with $\supp(g)\subsetneq \supp(f_h)$; it is easily seen that this would contradict the fact that $\{f_1,\ldots,f_r\}$ is reduced, whether $\ini_{\prec}(g)=\ini_{\prec}(f_h)$ or not.  We have thus proven the following lemma.

\begin{lemma}\label{useful}
Let $I$ be a homogeneous ideal, $\prec$ a fixed monomial order and $G$  the reduced Gr\"obner basis of $I$ with respect to $\prec$. Then, $\supp(G)\subseteq\cs(I)$. 
\end{lemma}

Let now ${\bf y}=(y_{ij})_{i,j=1,\ldots,n}$ be a matrix of indeterminates and $K({\bf y})$ an extension field of $K$. 
In the following we shall denote by  $\gamma$ the $K$-algebra homomorphism 
\begin{equation}\label{gamma}\gamma: \; K({\bf y})[X_1,\ldots,X_n] \longrightarrow K({\bf y})[X_1,\ldots,X_n], \;\;\;\;\;\;\;\; \gamma X_i\mapsto \sum_{j=1}^n y_{ij} X_j \hbox{\;\;\;\;for all\;\;} i=1,\dots,n.
\end{equation}

\begin{definition}\label{gcs}
Let $I$ be a homogeneous ideal of $A$. We define the {\em generic circuits set} of $I$ as  $\cs(\gamma I)$, and we  denote it by
$\gcs(I)$.  
Given a non-negative integer $d$, we let the {\em generic circuits set of $I$ truncated  at} $d$, denoted by $\gcs(I_{\leq d})$, be the circuits set $\cs(\gamma I_{\leq d})$. 
\end{definition}

\noindent 
It is easy to see that $\gcs(I_{\leq d})=\cs((\gamma I)_{\leq d})=\cs(\gamma I)_{\leq d}=\gcs(I)_{\leq d}$.

\begin{remark}\label{invariant}
The generic circuits set of a homogeneous ideal $I$ is invariant under coordinates changes, i.e., for all $h\in\GL_n(K)$ one has
$\gcs(hI)=\gcs(I).$ To this end, observe that if ${\bf z}$ is the matrix associated with $\gamma h$ then $K({\bf y})$ and $K({\bf z})$ are the same field; in particular, the entries of ${\bf z}$ are algebraically independent over $K$. Moreover, $\gamma hI$ is the image of $I$ under the map $K({\bf z})[X_1,\ldots,X_n]\longrightarrow K({\bf z})[X_1,\ldots,X_n]$, with $X_i\mapsto \sum_{j=1}^n z_{ij} X_j$ for all $i=1,\dots,n.$ Hence, $\supp(\gamma f)=\supp(\gamma h f)$ for all $f\in A$.
\end{remark} 

\begin{notation}\label{rango}{\em
Let $S$ be a finite set of monomials of $A$ and $W$ a $K$-vector subspace of $A.$ We set $$\rk_S W:=\dim_K( W+\langle S \rangle)/W, \hbox{\;\;\;} \rk^{S}W:=\dim_K( W+\langle S \rangle)/\langle S \rangle.$$  
Evidently, $\rk^{S}W= \rk_S(W) + \dim_K W - \dim_K \langle S \rangle.$}
\end{notation}

\begin{remark}\label{Srango}
  It is an easy observation completing the discussion before Lemma \ref{useful} that the circuits set  of a homogeneous vector space can be determined using ranks: Given a $K$-vector space  $W\subseteq A_d$,  a set $S$ is an element of  $\cs(W)$ if and only if $\rk_S W < |S|$ and  $\rk_{S'}(W)=|S'|$ for all $\emptyset \neq S' \subsetneq S$. 
\end{remark}

Let $W$ be a $K$-vector subspace of $A_d$ with basis $B.$ Consider an ordered monomial basis of $A_d,$ and let $M_W$  be the $\dim_K W\times \dim_K A_d$ matrix  whose   $(i,j)^{\rm th}$ entry  is  the coefficient of the $j^{\rm th}$-monomial in the $i^{\rm th}$ basis element of $W.$ Clearly, a minor of  $M_W$ is an element  of $K$.\\
\noindent
Now we consider $\gamma W$ together with its basis $\gamma B$, where $\gamma$ is as in \eqref{gamma}. The minors of the 
 matrix $M_{\gamma W}$  are polynomials in $K[{\bf y}]$ which specialize to the  minors of $M_W$ when all of the $y_{ij}$ are evaluated at $1$ if $i=j$ and at $0$ otherwise.
 
 When $S$ is a set of monomials in $A_d$ we may thus conclude that  $\rk_SW\leq \rk_S \gamma W$; if $K$ is infinite,  then there exists a non-empty Zariski open set $U \subseteq \GL_n(K) \subset K^{n^2}$ such that if $g\in U$ then 
\begin{equation}\label{Mrango}
\rk_S gW=\rk_S \gamma W=\max\{\rk_S hW \: h\in\GL_n(K)\}, \hbox{\;\;\;for all\;\;\;} S\subseteq A_d.
\end{equation}
\begin{theorem}\label{generalcs}
Let $K$ be an infinite field,  $d$ a positive integer and $I\subseteq A$ a homogeneous ideal. Then, there exists a non-empty Zariski open set $U\subseteq \GL_n(K)\subset K^{n ^2}$ such that $\gcs(I_{\leq d})=\cs(gI_{\leq d})$ for all $g\in U$.
\end{theorem}
\begin{proof}
For any integer $i$, a set $S$ belongs to $\cs(\gamma I_i)$ if and only if the condition on ranks of Remark \ref{Srango}  holds for the $K$-vector space $\gamma I_i$ or for $gI_i$, where $g$ belongs to a non-empty Zariski open set, say $U_i$, for which \eqref{Mrango} holds. The desired open set $U$ can be taken simply as the intersection of all $U_i$, $i=0,\ldots,d$.
\end{proof}

\noindent
We do not know at the present time whether there exists a non-empty Zariski open set $U$ such that $\gcs(I)=\cs(g I)$ for all $g\in U$, or whether there exists any such $g$ at all.

Let $I$ be a homogeneous ideal of $A$ and consider now the following equivalence relation on $\RR^n$: two weights $\omega$ and $\omega'$ are said to be equivalent  if and only if  $\ini_{\omega}(I)=\ini_{\omega'}(I)$. The closures with respect to the Euclidean topology of such equivalence classes are convex polyhedral cones and the collection of all such cones  form a fan, which is called the {\em Gr\" obner fan of $I$}, see \cite{MoRo}, \cite{St}. 

\begin{notation}\label{indiunrobo}{\em
    Let $S$ be a finite set of monomials of $A$ and $\omega$ a weight. We denote by $\ini_\omega(S)$ the set of all elements of $S$ with maximal weight. Similarly, for a collection $T$ of finite sets of monomials we denote by $\ini_{\omega}(T)$ the set of all $\ini_\omega(S)$ for $S$ in $T.$ When $I$ is a subset of $A$  we  will refer to $\ini_\omega (\cs(I))$ as the {\em initial circuits set of $I$ with respect to $\omega$}. 
It is not hard to see that  when $I$ is a homogeneous ideal $\ini_\omega (\cs(I))= \cs(\ini_\omega(I)),$ and also that  $\ini_\omega( \cs(I_{\leq d}))= \cs((\ini_\omega(I))_{\leq d}).$}    
    
\end{notation}

The following technical result will be useful in the remaining part of the section. See also Proposition 2.3 in \cite{St} and Proposition 2.6 in \cite{FuJeTh}.

\begin{lemma}\label{pretim}
  Let $I$ be a homogeneous ideal of $A$,  $\omega, \omega'\in\RR^n$ be two weights and $\prec$ a given monomial order. If $\{f_1,\ldots,f_r\}$ is a reduced Gr\"obner basis of $I$ with respect to $\prec_\omega$, then 
$$\ini_\omega(I)=\ini_{\omega'}(I) \hbox{\;\;\; if and only if \;\;\;} \ini_\omega(\supp(f_i))=\ini_{\omega'}(\supp(f_i)) \hbox{\;\;\; for  \;\;\;} i=1,\ldots,r.$$
\end{lemma} 
\begin{proof}
We first assume that $\ini_\omega(I)=\ini_{\omega'}(I).$ Since $\ini_\prec(\ini_\omega(I))=\ini_{\prec_\omega}(I)$,  we immediately have that  $\ini_{\prec_\omega}(I)=\ini_{\prec_{\omega '}}(I)$, hence $\ini_{\prec_{\omega '}}(f_i)\in\ini_{\prec_\omega}(I)$. Since $\ini_{\prec_{\omega '}}(f_i)$ is a monomial of $f_i$ and $\{f_1,\ldots,f_r\}$ a reduced Gr\"obner basis with respect to $\prec_\omega$, this implies that  $\ini_{\prec_\omega}(f_i)=\ini_{\prec_{\omega '}}(f_i)$ for $i=1,\ldots,r$.  Now, we assume by contradiction that $\ini_\omega(\supp(f_i))\neq \ini_{\omega '}(\supp(f_i))$ for some $i$, and accordingly $\ini_\omega(f_i)-\ini_{\omega '}(f_i)$ is a non-zero element of $\ini_\omega(I)$ which does not contain $\ini_{\prec_\omega}(f_i)$ in its support. Thus,  $\ini_\prec(\ini_\omega(f_i)-\ini_{\omega '}(f_i))\in \ini_{\prec_\omega}(I)$ contradicts the fact that $\{f_1\ldots,f_r\}$ is a reduced Gr\"obner basis of $I$ with respect to $\prec_\omega$.\\
Vice versa, when $\ini_\omega(\supp(f_i))=\ini_{\omega '}(\supp(f_i))$ then $\ini_{\omega}(I)=(\ini_\omega(f_i): i=1,\ldots,r)=(\ini_{\omega '}(f_i): i=1,\ldots,r)\subseteq \ini_{\omega '}(I)$, and equality is forced by the Hilbert function.
\end{proof}

From the previous result it follows as a corollary that the set of all supports of all reduced Gr\"obner bases of a homogeneous ideal $I$ determines the equivalence relation on weights that defines the Gr\"obner fan of $I$. 

\begin{corollary}\label{cervellobollito}
Let $I$ and $J$ be homogeneous ideals of $A$ with canonical universal Gr\"obner bases $G_1$ and $G_2$ respectively. If $\supp(G_1)=\supp(G_2)$ then $\Gf(I)=\Gf(J)$. 
\end{corollary}

\begin{proposition}
Let $I$ and $J$ be homogeneous ideals with same Hilbert function, $\prec$ a fixed monomial order, $G_1$ and  $G_2$  the reduced Gr\"obner bases (with respect to $\prec$)  of $I$ and $J$ respectively.  If $d$  is an integer 
greater than or equal to the largest degree of an element of $G_1$ and $\cs(I_{\leq d})=\cs(J_{\leq d})$ then 
$\supp(G_1)=\supp(G_2)$.
\end{proposition}
\begin{proof} Let $G_1=\{f_1,\ldots,f_r\}.$
We know that  $G_1\subseteq I_{\leq d}$, and by Lemma \ref{useful} $\supp(G_1)\subseteq \cs(I_{\leq d})=\cs(J_{\leq d})$. Thus, there exists a subset $H_2=\{h_1,\ldots,h_r\}$ of $J$ with $\supp(h_i)=\supp(f_i)$ 
for $i=1,\ldots,r$, and we may assume that the initial monomials  of $h_1,\ldots,h_r$ with respect to $\prec$ have coefficients equal to $1.$
Now,   $(\ini_{\prec}(h_i) \: i=1,\ldots, r)$ has the same Hilbert function as $I$ and, thus, as $J$; consequently, $H_2$ is a Gr\"obner basis of  $J$  and it is clearly reduced. Since such a basis is unique, $H_2=G_2$ and we are done.
\end{proof}

\begin{corollary}\label{preTim}
  Let $I$ and $J$ be homogeneous ideals  with the same Hilbert function,  and let $d$ be the largest degree of a minimal generator of the lex-segment ideal with same Hilbert function as $I$ and $J$.
If $\cs(I_{\leq d})=\cs(J_{\leq d})$ then $\Gf(I)=\Gf(J)$.
\end{corollary}

As a special case of the above corollary, we now recover one of the main results of  \cite{R\"oSc}, namely Corollary 3.2, where the existence of generic Gr\"obner fans is proven.

\begin{corollary}\label{tim}
Let $I$ be a homogeneous ideal of $A$. Then, there exists a non-empty Zariski open set $U\subseteq\GL_n(K)$ such that 
$\Gf(gI)=\Gf(hI)$ for all $g, h\in U$.
\end{corollary} 
\begin{proof}
It is a direct consequence of  the previous corollary and Theorem \ref{generalcs}.
\end{proof}

\section{General initial ideals with respect to weights}

As we have explained in the introduction, the generic initial ideal of $I$ with respect to a given monomial order $\prec$ is an important invariant of $I$ and one would like to have a similar invariant when using  weights.  There are two naive definitions of what a generic initial ideal with respect to a given weight could be. One might set 
$\gin_{\omega}(I)$ to be $\ini_\omega(\gamma I)\subseteq K({\bf y})[X_1,\ldots,X_n]$, where $\gamma$ is as in \eqref{gamma} and Definition \ref{gcs}. The disadvantage here is that to define $\gin_{\omega}(I)$ in this way  would not provide a coordinate-independent invariant, as the following easy example shows. If we take $I=(X_1)\subseteq K[X_1,X_2]$,  a coordinates change such that  $X_1 \mapsto X_1+X_2$ and $\omega=(1,1)$ then we have
$\ini_{\omega}(X)=(y_{11}X_1+y_{12}X_2)\neq ((y_{11}+y_{21})X_1+(y_{12}+y_{22})X_2)=\ini_{\omega}(\gamma(X_1+X_2))$.
Moreover, with this definition, the resulting generic initial ideal could not be viewed as an ideal of $K[X_1,\ldots,X_n]$. Alternatively, when $K$ is infinite one might be tempted to let $\gin_{\omega}(I)$ be $\ini_{\omega}(gI)$ for a general $g\in \GL_n(K)$, but this is in some sense meaningless because it relies on the existence of a non-empty Zariski open set $U$ where $\ini_\omega (hI)$ is constant for all $h\in U$. In the above example, clearly such an open set does not exist.
Therefore, we would like to emphasize that the expression {\em a general initial ideal with respect to} $\omega$ should be used in the same way as {\em a general linear form} or {\em a general hyperplane section} is used: always in combination with a specific and well-defined property $\mathcal{P}$ of $\ini_{\omega}(gI)$ which is constant on a Zariski open set of $\GL_n(K).$ Keeping this in mind, it is correct to  phrase Theorem \ref{stab} in the following way: {\em a general initial ideal with respect to $\omega$ is fixed under the action of the group $B_\omega.$}
\noindent
We have seen in the previous section how the generic circuits set of an ideal $I$ can be used as an invariant of $I$. Also, rather than defining the generic initial ideal of $I$ with respect to a weight, one can consider the generic initial circuits set $\gcs(\ini_\omega(I))=\ini_\omega(\gcs(I))$  and its truncations $\gcs(\ini_\omega(I)_{\leq d})$ at some $d\in\NN$, see Notation \ref{indiunrobo}.  Remark \ref{invariant} yields that any such $\gcs(\ini_\omega(I)_{\leq d})$ is an invariant of $I$.
Moreover, when $K$ is infinite a truncated generic circuits set of $I$ can be defined using general changes of coordinates: One can let $\gcs(\ini_{\omega}(I_{\leq d}))$  be  $\ini_{\omega}(\cs(gI_{\leq d}))$  for a general change of coordinates $g$ and this makes sense because of Theorem \ref{generalcs}. 

\subsection{The subgroup $\mathcal{B}_{\omega}$ of the Borel group} In the usual setting,  generic initial ideals in any characteristic have the property of being fixed under the action of  the Borel subgroup $\mathcal{B}$ of $\GL_n(K)$ consisting of invertible $n\times n$ upper-triangular matrices.
 Let $\omega=(\omega_1,\ldots,\omega_n) \in \ZZ^n$  be a fixed weight. By  relabeling the variables if necessary, 
we will further assume that $\omega_1\geq\cdots\geq \omega_n$. 
 We define a subset  of $\mathcal{B}$ by taking all $n\times n$ matrices $M=(m_{ij})$ in $\mathcal{B}$ such that $m_{ii}=1$ for $i=1,\ldots, n$ and $m_{ij}=0$ if $\omega_i=\omega_j$, and denote it by $\mathcal{B}_\omega$. Obviously, the identity matrix belongs to $\mathcal{B}_\omega$. If $M, N\in \mathcal{B}_\omega$ and $\omega_i=\omega_j$ for $i\not=j$, then the $(i,j)^{\rm th}$ entry of $MN$ is zero;  it is also easy to verify that every $M\in\mathcal{B_\omega}$ has an inverse in $\mathcal{B_{\omega}}$ by computing the row-echelon form of $M$ augmented with the identity. We have thus verified that  $\mathcal{B}_\omega$ is a subgroup of $\mathcal{B}$. The main result we want to prove next, and we do in Theorem \ref{stab},  is that  general initial ideals with respect to a non-negative weight $\omega$ are fixed under the action of $\mathcal{B_\omega}.$

 We now let $d$ be a positive integer and, as before, let $A_d$ be the degree $d$ part of the polynomial ring $A$. The largest weight of a monomial in $A_d$ is $\omega_1d.$ For all $0\leq a\leq \omega_1d+1$ we let $S_a$  be the set of all monomials of $A_d$ of weight strictly less then $a$. 
\noindent 
Given a vector subspace $W$ of $A_d$, we let $\alpha_\omega(W)$ denote the vector
$$\alpha_{\omega}(W):=(\rk^{S_{\omega_1d}}W, \rk^{S_{\omega_1d-1}}W,\ldots,  \rk^{S_{1}}W).$$
Clearly,  when $V, W$ are $K$-vector subspaces of $A_d$ and $\cs(V)=\cs(W)$ then $\alpha_{\omega}(V)=\alpha_{\omega}(W),$ by Remark \ref{Srango}.  Next, we write $\alpha_{\omega}(V) \geq \alpha_{\omega}(W)$ when the inequality holds pointwise; when in addition $\alpha_{\omega}(V) \not = \alpha_\omega(W)$ we write $\alpha_\omega(V) > \alpha_\omega(W)$.

Recall that a $K$-vector subspace $W$ of $A_d$ is homogeneous with respect to $\omega$ if it is spanned by polynomials which are homogeneous with respect to $\omega$. 

\begin{proposition}\label{borelranks} Let $W$ be a $K$-vector subspace of $A_d$  and $\omega=(\omega_1,\ldots,\omega_n)$ be a weight with $\omega_1\geq \cdots \geq \omega_n\geq 0.$ 
 Then, for every integer $0\leq a \leq \omega_1d$,  the dimension of the homogeneous component of $\ini_\omega(W)$ of weight $a$ is  $\rk^{S_a}W-\rk^{S_{a+1}}W$. Furthermore, if $W$ is homogeneous with respect to $\omega$ and $b\in\mathcal{B}$ is an upper-triangular change of coordinates then  $\alpha_\omega(bW)\geq\alpha_\omega(W).$
\end{proposition}
\begin{proof}  Let $\prec$ be a monomial order. We consider, as we did to prove \eqref{Mrango}, the matrix $M_W$ associated with $W$ after having ordered a monomial basis of $A_d$ by means of $\prec_{\omega}.$ The first part of the statement can be verified by  computing the row-echelon form of $M_W$. The desired inequality follows from the definition of $\alpha_\omega$, since the image under $b$ of a monomial  is the sum of that monomial and a linear combination of other monomials of equal or greater weight.
\end{proof}

\begin{theorem}  \label{stab}
Let $I$ be a homogeneous ideal of $A=K[X_1,\ldots,X_n]$, with $|K|=\infty$. Let also $\omega=(\omega_1,\ldots,\omega_n)$ be a weight with $\omega_1\geq \cdots \geq \omega_n.$  Then, a general initial ideal of $I$ with respect to $\omega$ is  $\mathcal {B}_\omega$-fixed, i.e., 
there exists a non-empty Zariski open set $U$ such that $b\left(\ini_\omega(g(I)\right)=\ini_\omega(g(I))$, for all $g\in U$ and all $b\in\mathcal{B}_\omega$.
\end{theorem}

\begin{proof}
 Observe that, if  $f$ is a homogeneous polynomial, $\omega$ a weight and we let $\omega':=\omega + (1,1,\dots,1)$, then $\ini_\omega(f)=\ini_{\omega'}(f)$; therefore we may assume that $\omega\in\RR_{\geq 0}^n$. Now notice that there exists an upper bound $D$ for the generating degrees of all the ideals  with the same Hilbert function as $I$. Thus, if $\omega$ and $\omega'$ induce the same partial order on all monomials of degree less than or equal to $D$, then $\ini_{\omega}(J)=\ini_{\omega'}(J)$ for every homogeneous ideal $J$ with such  Hilbert function. Hence, we may further assume that $\omega\in\ZZ^n$, with $\omega_1\geq\ldots\geq\omega_n\geq 0$.

\noindent 
By Theorem \ref{generalcs}, we let $U$ be a non-empty Zariski open set such that $\gcs(I_{\leq D})=  \cs(gI_{\leq D})$ for all $g\in U.$
Clearly, it is enough to prove the equality degree by degree up to the degree $D.$
Let $W$ be the degree $d$ component of $\ini_\omega(gI)$ with $0\leq d\leq D.$
First, we decompose $A_d$ as a direct sum of vector spaces $V_p\oplus V_{p-1}\oplus\cdots\oplus V_0$, where $V_i$ is generated by all polynomials in $A_d$ which are homogeneous with respect to $\omega$ and of weight $i.$ 
Accordingly, $p=d\omega_1$, $V_p=K[X_1,\ldots,X_j]_d$ with $j=\max\{i \: \omega_i=\omega_1\},$ and $b$ acts as the identity on $V_p$. Now, if we decompose $W$ in an analogous manner as $\oplus_{i=p}^0W_i$, we immediately have that $W_p\subseteq V_p$ and $W_p$ is fixed under the action of $b$. 
We may thus assume by induction that $b\left(\oplus_{i< j} W_{p-i}\right)=\oplus_{i< j} W_{p-i}$  and we need to prove that $b\left(\oplus_{i\leq j} W_{p-i}\right)$ is equal to $\oplus_{i\leq j} W_{p-i}$. In order to see this it is enough to prove a containment.  Assume $W_j\not = 0$, otherwise the result follows from the inductive hypothesis. When $0\neq f\in W_j$  we can write $b(f)$ as $f+q$, where $q$ is a sum of polynomials of weight larger than $j$, i.e., $q\in V_p\oplus\cdots \oplus V_{j+1},$ and it is left  to show  that $q\in  W_p\oplus\cdots \oplus W_{j+1}$.  If this would not be the case, then $\ini_{\omega}(f)=\ini_{\omega}(q)\not \in W_p\oplus\cdots \oplus W_{j+1}.$ Thus, by Proposition 
\ref{borelranks} we would have $\rk^{S_{j+1}}(b(W))> \rk^{S_{j+1}}(W)$ and in particular   
$\alpha_{\omega}(bW)>\alpha_{\omega}(W)$.

  On the other hand, by definition of $W$ and Subsection \ref{flatt} we have that 
$$\alpha_{\omega}(bW)=\alpha_{\omega}(b(\ini_\omega(gI)_d))=\alpha_{\omega}\left(b(\widetilde{(gI)}_{t=0})_d\right),$$ where $\widetilde{gI}$ denotes the homogenization of the ideal $gI$ with respect to $\omega$. 
By arguing as before \eqref{Mrango} we see that this quantity is determined by the ranks of the non-zero minors of the matrix $M_{b(\widetilde{(gI)}_{t=0})_d}$, each such minor corresponding to a non-zero minor of the matrix $M_{b(\widetilde{(gI)})_d}$, which has entries in $K[t]$.   In particular, we can find an $a\in K$ such that all the non-zero minors of $M_{b(\widetilde{(gI)})_d}$ do not vanish after applying the  substitution $t=a$.
Hence, $\alpha_{\omega}(bW) \leq  \alpha_{\omega}\left(b((\widetilde{(gI)}_{t=a})_d\right)=\alpha_{\omega}\left(b(D_{a}(gI))_d\right)$.  Accordingly,  
$$\alpha_{\omega}(bW)\leq \max\{\alpha_{\omega}(hI_d) 
\: h\in\GL_n(K)\}=\alpha_{\omega}(gI_d),$$ where the last equality follows from the choice of $g\in U$ and \eqref{Mrango}.  
Proposition \ref{borelranks} implies that $\alpha_{\omega}(gI_d)= \alpha_{\omega}(\ini_{\omega}((gI_d)),$ which is by definition $\alpha_{\omega}(W)$. We have thus obtained that $\alpha_{\omega}(bW)\leq\alpha_{\omega}(W)$ and the desired contradiction. 
\end{proof}

\noindent
If we consider a homogeneous ideal  $I$ of $A$ and  take for instance $\omega$ to be the weight $(1,1,\dots,1,0)$, then by the previous theorem a general initial ideal of $I$ with respect to $\omega$  is fixed under any coordinates change which is the identity on $X_1,\dots,X_{n-1}$; this fact can be useful in applications, see for instance \cite{CaKu} Section 4 and \cite{CaSb} Proposition 1.6.


\end{document}